\documentclass[12pt,reqno]{amsart}

\usepackage{amssymb}
\usepackage{amsthm}
\usepackage{amsaddr}

\usepackage{geometry}
\newtheorem{theorem}{Theorem}
\newtheorem*{lemma}{Lemma}
\newtheorem*{corollary}{Corollary}

\newcommand{\R}{\mbox{$\mathbb{R}$}}
\newcommand{\C}{\mbox{$\mathbb{C}$}}
\newcommand{\F}{\mbox{$\mathbb{F}$}}

\begin{document}

\title[Continuity and Discontinuity of Seminorms]
{Continuity and Discontinuity of Seminorms\\ on Infinite-Dimensional Vector Spaces}

\author[Jacek Chmieli\'nski]{Jacek Chmieli\'nski}
\address{Department of Mathematics, Pedagogical University of Krak\'ow\\ Krak\'{o}w, Poland\\
E-mail: jacek.chmielinski@up.krakow.pl}
%\email{jacek.chmielinski@up.krakow.pl}
\author[Moshe Goldberg]{Moshe Goldberg$^{1}$}\thanks{$^{1}$Corresponding author}
\address{Department of Mathematics, Technion -- Israel Institute of Technology\\ Haifa, Israel\\
E-mail: mg@technion.ac.il}
%\email{mg@technion.ac.il}

\begin{abstract}
Let $S$ be a seminorm on an infinite-dimensional real or complex vector space~$X$. Our purpose in
this note is to study the continuity and discontinuity properties of $S$ with respect to certain norm-topologies
on $X$.
\end{abstract}


\maketitle

Throughout this note let $X$ be a vector space over a field $\F$, either $\R$ or $\C$. As usual, a~real-valued function $N\colon X\to\R$ is said to be a {\em norm} on $X$ if, for all $x,y\in X$ and $\alpha\in\F$,
$$
\begin{array}{l}
N(x)>0,\quad x\neq 0,\\
N(\alpha x)=|\alpha| N(x),\\
N(x+y)\leq N(x)+N(y).
\end{array}
$$

Further, a real-valued function $S\colon X\to\R$ is a {\em seminorm} on $X$ if, for all $x,y\in X$ and $\alpha\in\F$,
$$
\begin{array}{l}
S(x)\geq 0,\\
S(\alpha x)=|\alpha| S(x),\\
S(x+y)\leq S(x)+S(y);
\end{array}
$$
hence, a seminorm is a norm if and only if it is positive-definite.

We recall that if $X$ is finite-dimensional, then all norms on $X$ are equi\-valent, hence inducing on $X$ a unique norm-topology. This well-known fact leads to the following result.

\begin{theorem}[\cite{G}]\label{t1}
Let $S$ be a seminorm on a finite-dimensional vector space $X$ over $\F$, either $\R$ or $\C$. Then $S$ is continuous with respect to the unique norm-topology on $X$.
\end{theorem}

Contrary to the finite-dimensional case, if $X$ is infinite-dimensional then not all norms on $X$ are equivalent, and consequently we no longer have a~unique norm-topology. Indeed, our main purpose in this
note is to explore the questions of continuity and discontinuity of seminorms when the assumption of the finite-dimensionality is removed.

We begin our journey by posting the following lemma.

\begin{lemma}
Let $X$, an infinite-dimensional vector space over $\F$, either $\R$ or $\C$, be equipped with a seminorm $S$ and a norm $N$. Then:
\begin{enumerate}
\item[{\rm (a)}]
Continuity of $S$ at the origin implies ubiquitous continuity with respect to the topology induced by $N$.
\item[{\rm (b)}]
Discontinuity of $S$ at the origin implies ubiquitous discontinuity with respect to the above mentioned topology.
\end{enumerate}
\end{lemma}
\begin{proof}
The assertion in (a) is well known to many, old and young. Yet, for the reader's convenience, we provide a short proof.

Let $S$ be continuous at $x=0$ with respect to the topology induced by the given norm; hence, if $\{x_n\}_{n=1}^{\infty}$ is any sequence in $X$, then
\begin{equation}\label{e1}
N(x_n)\to 0\quad \mbox{implies}\quad S(x_n)\to 0\quad \mbox{as}\ n\to\infty.
\end{equation}
Choosing $x\in X$, and taking an arbitrary sequence $\{y_n\}_{n=1}^{\infty}$ for which $N(x-y_n)\to 0$, \eqref{e1}~implies that,
$$
S(x-y_n)\to 0\quad \mbox{as}\ n\to\infty.
$$
Thus, aided by the familiar inequality
\begin{equation}\label{e2}
|S(x)-S(y)|\leq S(x-y),\qquad x,y\in X,
\end{equation}
which follows from the subadditivity of $S$, we get,
$$
|S(x)-S(y_n)|\to 0,
$$
and part (a) of the lemma is in the bag.

As for part (b), the assumption that $S$ is discontinuous at $x=0$ implies that there exists a sequence
$\{x_n\}_{n=1}^{\infty}$ in $X$ such that
$$
N(x_n)\to 0\quad \mbox{but}\quad S(x_n)\not\to 0\quad \mbox{as}\ n\to\infty.
$$
Since $S(x_n)\not\to 0$ and $S$ is nonnegative, we can provide, by passing if necessary to a subsequence, a positive
constant $\kappa$ such that
$$
S(x_n)\geq\kappa,\qquad  n =1,2,3,\ldots .
$$
Moreover, since $N(x_n)\to 0$, by passing to yet another subsequence if necessary, we may assume that
$$
N(x_n)\leq \frac{1}{n^2},\qquad n=1,2,3,\ldots .
$$

Now, put
$$
y_n=nx_n,\qquad n=1,2,3,\ldots .
$$
Since $N(y_n)\leq\frac{1}{n}$ and $S(y_n)\geq n\kappa$, we see that
$$
N(y_n)\to 0\quad \mbox{but}\quad S(y_n)\to\infty\quad \mbox{as}\ n\to\infty.
$$
It follows that at any point $x\in X$, the sequence $\{z_n\}_{n=1}^{\infty}$ defined by
$$
z_n = x + y_n,\qquad n=1,2,3,\ldots,
$$
satisfies
$$
N(x-z_n)=N(y_n)\to 0.
$$
On the other hand,
$$
S(x-z_n)=S(y_n)\to\infty,
$$
which forces the discontinuity of $S$ at $x$.
\end{proof}
From the above lemma we may deduce the following result which hardly requires a proof.

\begin{theorem}\label{t2}
Let $X$, an infinite-dimensional vector space over $\F$, either $\R$ or $\C$, be equipped with a seminorm $S$ and a norm $N$. Then:
\begin{enumerate}
\item[{\rm (a)}]
$S$ is ubiquitously continuous in $X$ with respect to the topology induced by $N$ if and only if there exists some point of $X$ at which $S$ is continuous.
\item[{\rm (b)}]
Similarly, $S$ is ubiquitously discontinuous in $X$ with respect to the above mentioned topology if and only if there exists some point of $X$ at which $S$ is discontinuous.
\end{enumerate}
\end{theorem}

We turn now to the main purpose of this note, which is to show that given a nontrivial seminorm $S$ on $X$, there exist two different norms, such that $S$ is ubiquitously continuous in $X$ with respect to one of them,
and ubiquitously discontinuous with respect to the other.

\begin{theorem}\label{t3}
Let $S\neq 0$ be a seminorm on an infinite-dimensional vector space $X$ over $\F$, either $\R$ or $\C$. Then:
\begin{enumerate}
\item[{\rm (a)}]
There exists a norm with respect to which $S$ is ubiquitously continuous in $X$.
\item[{\rm (b)}]
There exists a norm with respect to which $S$ is ubiquitously dis\-con\-ti\-nuous in $X$.
\end{enumerate}
\end{theorem}
\begin{proof}
(a) Let $H$ be a Hamel basis for $X$. Then,
$$
B=\left\{\frac{h}{S(h)}:\, h\in H,S(h)>0\right\}\cup\{h\in H:\, S(h)=0\}
$$
is a basis for $X$ as well. So every $x$ in $X$ assumes a unique representation of the form
\begin{equation}\label{e3}
x=\sum_{b\in B}\alpha_{b}(x) b,\qquad \alpha_{b}(x)\in\F,
\end{equation}
where $\{b\in B:\, \alpha_{b}(x)\neq 0\}$ is a finite set.
With this representation at hand, we can easily confirm that the real-valued function
\begin{equation}\label{e4}
N(x)=\sum_{b\in B}|\alpha_{b}(x)|,\qquad x\in X,
\end{equation}
is a norm on $X$. Furthermore, by the construction of $B$, we infer that
$$
S(b)\in\{0,1\}\quad \mbox{for all}\ b\in B.
$$
Consequently, for every $x$ in $X$,
\begin{equation}\label{e5}
S(x)=S\left(\sum_{b\in B}\alpha_{b}(x) b\right)\leq \sum_{b\in B}|\alpha_{b}(x)|S(b)\leq \sum_{b\in B}|\alpha_{b}(x)|=N(x).
\end{equation}
By \eqref{e2} and \eqref{e5}, therefore,
$$
|S(x)-S(y)|\leq N(x-y),\qquad x,y\in X,
$$
and (a) follows.

(b) Consider the nonempty set
$$
A=\{x\in X:\, S(x)>0\},
$$
and let us show that
\begin{equation}\label{e6}
{\rm span}\,A = X.
\end{equation}
Indeed, assume on the contrary that ${\rm span}\,A\varsubsetneq X$. Then there exists an element $y\in X$ such that $y\not\in {\rm span}\,A$
and $S(y)=0$. Now, for every $z$ in $A$, we have
$$
S(y+z)=S(y+z)+S(y)\geq S(y+z-y)=S(z)>0,
$$
whence, both $z$ and $y+z$ belong to $A$, so $y=(y+z)-z\in {\rm span}\,A$, a~contradiction.

In view of \eqref{e6}, we may appeal to Corollary 4.2.2 in \cite{K} (compare \cite{H}), which tells us that $A$ contains a
Hamel basis $H$ for $X$. So by the definition of $A$, we have $S(h)>0$ for all $h\in H$. With this in mind, we fix
a sequence $\{h_n\}_{n=1}^{\infty}$ of distinct elements in $H$, and construct a new basis $B'$ for $X$ by replacing $h_n$ by
$$
b_n=\frac{nh_n}{S(h_n)},\qquad n=1,2,3,\ldots,
$$
and leaving all other elements in $H$ unchanged.

Surely,
$$
S(b_n)=n,\qquad n=1,2,3,\ldots .
$$
Further, modifying the representation in \eqref{e3} and the norm $N$ in \eqref{e4} by replacing the underlying basis $B$ by the new basis $B'$, we get,
$$
N(b_n)=1,\qquad n =1,2,3,\ldots .
$$
Hence, setting
$$
x_n=\frac{b_n}{n},\qquad n=1,2,3,\ldots,
$$
we obtain
$$
N(x_n)=\frac{1}{n}\to 0,
$$
whereas
$$
S(x_n)=1,\qquad n=1,2,3,\ldots .
$$
This establishes the discontinuity of $S$ at $x=0$; so by Theorem \ref{t2}(b), $S$ is discontinuous everywhere, and we are done.
\end{proof}
Falling back on Theorems \ref{t1} and \ref{t3}(b), we may record now the following simple observation.

\begin{corollary}
Let $S\neq 0$ be a seminorm on a vector space $X$ over $\F$, either $\R$ or $\C$. Then $S$ is continuous with respect to every norm on $X$ if and only if $X$ is finite-dimensional.
\end{corollary}

We conclude this note by recalling that a seminorm $S$ is {\em left-equiva\-lent} to a norm $N$ on $X$ if there exists a constant $\tau > 0$ such that
$$
S(x)\leq\tau N(x)\quad \mbox{for all}\ x\in X.
$$
Therefore, by \eqref{e5}, {\em every seminorm $S$ on $X$ is left-equivalent to some norm on $X$}; for example, to the norm in \eqref{e4} (which, of course, depends on $S$). As indicated in \cite{G}, {\em if $X$ is finite-dimensional, then $S$ is left-equivalent to every norm on X}.

%===============================================================

%=======================================================================
\end{document}